\documentclass[11pt]{amsart}
\usepackage{graphicx}
\usepackage[active]{srcltx}
 \makeatletter
\renewcommand*\subjclass[2][2000]{%
  \def\@subjclass{#2}%
  \@ifundefined{subjclassname@#1}{%
    \ClassWarning{\@classname}{Unknown edition (#1) of Mathematics
      Subject Classification; using '1991'.}%
  }{%
    \@xp\let\@xp\subjclassname\csname subjclassname@#1\endcsname
  }%
}
 \makeatother
\usepackage{enumerate,url,amssymb,  mathrsfs}

\newtheorem{theorem}{Theorem}[section]
\newtheorem{lemma}[theorem]{Lemma}
\newtheorem*{lemma*}{Lemma}

\theoremstyle{definition}

\theoremstyle{remark}
\newtheorem{remark}[theorem]{Remark}

\numberwithin{equation}{section}

\def\XXint#1#2#3{{\setbox0=\hbox{$#1{#2#3}{\int}$}
\vcenter{\hbox{$#2#3$}}\kern-.5\wd0}}

\def\le{\leqslant}
\def\ge{\geqslant}
\begin{document}

\title{Heinz-Schwarz inequalities for harmonic mappings in the unit ball}

\author[Kalaj]{David Kalaj}
\address{University of Montenegro, Faculty of Natural Sciences and
Mathematics, Cetinjski put b.b. 81000 Podgorica, Montenegro}
\email{davidkalaj@gmail.com}



\maketitle

\makeatletter\def\thefootnote{\@arabic\c@footnote}\makeatother

\begin{abstract}
We first prove the following generalization of Schwarz lemma for harmonic mappings. Let $u$ be a harmonic mapping of the unit ball onto itself. Then we prove the inequality $\|u(x)-(1-\|x\|^2)/(1+\|x\|^2)^{n/2} u(0)\|\le U(|x| N)$. By using the Schwarz lemma for harmonic mappings  we derive Heinz inequality on the boundary of the unit ball by providing a sharp constant $C_n$ in the inequality: $\|\partial_r u(r\eta)\|_{r=1}\ge C_n$, $\|\eta\|=1$, for every harmonic mapping of the unit ball into itself satisfying the condition $u(0)=0$, $\|u(\eta)\|=1$.
\end{abstract}

\maketitle

\section{Introduction}\label{intsec}
E. Heinz in his classical paper \cite{heinz}, obtained the following result:  If $u$ is a harmonic diffeomorphism of the unit disk $\mathbf{U}$ onto itself satisfying the condition $u(0)=0$, then $$|u_x(z)|^2+|u_y(z)|^2\ge \frac{2}{\pi^2},\, \, z\in \mathbf{U}.$$ The proof uses the following representation of harmonic mappings in the unit disk \begin{equation}\label{repres} u(z) = f(z)+\overline{g(z)},\end{equation} where $f$ and $g$ are holomorphic functions with $|g'(z)|<|f'(z)|$.  It uses the maximum principle for holomorphic functions and the following sharp inequality \begin{equation}\label{heinz}\liminf_{r\to 1^{-}}\left|\frac{\partial u(re^{it})}{\partial r}\right|\ge \frac{2}{\pi}\end{equation} proved by using the Schwarz lemma for harmonic functions. The aim of this paper is to generalize inequality \eqref{heinz} for several dimensional case.

If $u$ is a harmonic mapping of the unit ball onto itself, then we do not have any representation of $u$ as in \eqref{repres}.

It is well known
that a harmonic function (and a mapping)  $u\in L^\infty ({B}^n)$, where $B=B^n$ is the unit ball with the boundary $S=S^{n-1}$,  has the following integral representation
\begin{equation}\label{poi}u(x)=\mathcal{P}[f](x)=\int_{S^{n-1}}P(x,\zeta)f(\zeta)d\sigma(\zeta),\end{equation}
where
$$P(x,\zeta)=\frac{1-\|x\|^2}{\|x-\zeta\|^n}, \zeta\in S^{n-1} $$
is Poisson kernel and $\sigma$ is the unique
normalized rotation invariant Borel measure on $S^{n-1}$ and $\|\cdot \|$ is the Euclidean norm.

We have the following Schwarz lemma for harmonic mappings on the unit ball $B^n$ (see e.g. \cite{ABR}). If $u$ is a harmonic mapping of the unit ball into itself such that $u(0)=0$ then  \begin{equation}\label{shar}\|u(x)\|\le U(rN),\end{equation} where $r=\|x\|$, $N=(0,\dots,0,1)$ and $U$ is a harmonic function of the unit ball into $[-1,1]$ defined by \begin{equation}\label{uext}U(x)= \mathcal{P}[\chi_{S^+}-\chi_{S^-}](x),\end{equation} where $\chi$ is the indicator function and $S^+=\{x\in S: x_n \ge 0\},$ $S^-=\{x\in S: x_n \le 0\}.$  Note that, the standard harmonic Schwarz lemma is formulated for real functions only, but we can reduce the previous statement to the standard one by taking $v(x) = \left<u(x), \eta\right>$, for some $\|\eta\|=1$, where $\left<\cdot, \cdot\right>$ is the Euclidean inner product. Indeed, we will prove a certain generalization of \eqref{shar} without the a priory condition $u(0)=0$ (Theorem~\ref{scp}).   For Schwarz lemma for the derivatives of harmonic mappings on the plane and space we refer to the papers \cite{kh,kp}. It is worth to mention here a certain generalization of \eqref{heinz} for the mappings which are solution of certain elliptic partial differential equations in the plane \cite{chenvuo}. For certain boundary Schwarz lemma on the unit ball for holomorphic mappings in $\mathbf{C}^n$ we refer to the paper \cite{liu}.

 By using Hopf theorem it can be proved (\cite{pacific}) that if $u$ is a harmonic mapping of the unit ball onto itself such that $u(0)=0$ and $\|u(\zeta)\|=1$, then  $$\liminf_{r\to 1}\left\|\frac{\partial u}{\partial r}(r\zeta)\right\|\ge C_n,$$ where $C_n$ is a certain positive constant. Our goal is to find the largest constant $C_n$. This is done in Theorem~\ref{Theo1} and Theorem~\ref{Theo2}.
\section{Preliminaries and main results}
First we prove the following generalization of harmonic Schwarz lemma for $B^n$, $n\ge 3$. The case $n=2$ has been treated and proved by Pavlovic   \cite[Theorem~3.6.1]{pavlovic}.
\begin{theorem}\label{scp}
Let $u$ be a harmonic mapping of the unit ball onto itself, then \begin{equation}\label{geom}\left\|u(x)-\frac{1-\|x\|^2}{(1+\|x\|^2)^{n/2}}u(0)\right\|\le U(\|x\|N).\end{equation}
\end{theorem}
\begin{proof}
Assume first that $x=rN$. We have that
$$u(rN)= \int_{S^{n-1}} \frac{1-r^2}{\|\zeta-rN\|^n}f(\zeta)d\sigma(\zeta), $$ and so $$u(r N)-\frac{1-r^2}{(1+r^2)^{n/2}}u(0)=\int_{S^{n-1}} \left(\frac{1-r^2}{\|\zeta-rN\|^n}-\frac{1-r^2}{(1+r^2)^{n/2}}\right)f(\zeta)d\sigma(\zeta).$$
Further we have \[\begin{split}\|u(r N)-\frac{1-r^2}{(1+r^2)^{n/2}}u(0)\|&\le  \int_{S^{n-1}} \left|\frac{1-r^2}{\|\zeta-rN\|^n}-\frac{1-r^2}{(1+r^2)^{n/2}}\right|d\sigma(\zeta)
\\&=\int_{S^{+}} \left(\frac{1-r^2}{\|\zeta-rN\|^n}-\frac{1-r^2}{(1+r^2)^{n/2}}\right)d\sigma(\zeta)\\&+\int_{S^{-}} \left(\frac{1-r^2}{(1+r^2)^{n/2}}-\frac{1-r^2}{\|\zeta-rN\|^n}\right)d\sigma(\zeta).\end{split}\]

Thus  $$\left\|u(r N)-\frac{1-r^2}{(1+r^2)^{n/2}}u(0)\right\|\le U(rN).$$ Now if $x$ is not on the ray $[0,N]$, we choose a unitary transformation $O$ such that $O (N)=x/|x|$. Then we make use of harmonic mapping $v(y) = u(O(y))$ for which we have $v(rN)=u(O(rN))= u(x)$. By making use of the previous proof we obtain \eqref{geom}.
\end{proof}
\subsection{Hypergeometric functions}
In order to formulate and to prove our next results recall the basic definition of hypergeometric functions. For two positive integers $p$ and $q$ and vectors $a=(a_1,\dots, a_p)$ and $b=(b_1,\dots, b_q)$ we set $${_pF_q}[a;b,x] = \sum_{k=0}^\infty \frac{(a_1)_k\cdots (a_p)_k}{(b_1)_k \cdots (b_q)_k \cdot k!} x^k,$$ where $(y)_k: =\frac{\Gamma(y+k)}{\Gamma(y)}= y(y+1)\dots (y+k-1)$ is the Pochhammer symbol. The hypergeometric series converges at least for $|x|<1$. For basic properties and formulas concerning trigonometric series we refer to the book \cite{hyper}.
 The most important step in the proof of our main results i.e. of Theorem~\ref{Theo1} and Theorem~\ref{Theo2} below, is the following lemma
\begin{lemma}\label{granit} The function $V(r)= \frac{\partial U(rN)}{\partial r}$, $0\le r\le 1$ is decreasing on the interval $[0,1]$ and   we have
$$V(r)\ge V(1)=C_n:=\frac{  n! \left(1+n-(n-2)\, _2\mathrm{F}_1\left[\frac{1}{2},1,\frac{3+n}{2},-1\right]\right)}{2^{3 n/2} \Gamma\left[\frac{1+n}{2}\right] \Gamma\left[\frac{3+n}{2}\right]}.$$
\end{lemma}
\begin{proof}
By using spherical coordinates $\eta=(\eta_1,\dots, \eta_n)$ such that $\eta_n= \cos \theta$, where $\theta$ is the angle between the vector $x$ and $x_n$ axis, we obtain from \eqref{uext} that $$U(rN)= \frac{\Gamma\left[\frac{n}{2}\right]}{\sqrt{\pi } \Gamma\left[\frac{n-1}{2} \right]}\int_0^\pi \frac{(1-r^2)\sin^{n-2}\theta }{(1+r^2-2 r \cos\theta)^{n/2}} (\chi_{S^+}(x)-\chi_{S^-}(x))d\theta$$ and so
$$U(rN)=\frac{\Gamma\left[\frac{n}{2}\right]}{\sqrt{\pi } \Gamma\left[\frac{n-1}{2} \right]}\int_0^{\pi/2} \left(\frac{{(1-r^2)\sin^{n-2}\theta} }{(1+r^2-2 r \cos\theta)^{n/2}}-\frac{{(1-r^2)}\cos^{n-2}\theta }{(1+r^2+2 r \sin\theta)^{n/2}}\right)d\theta$$ or what can be written as
$$U(rN)=\frac{\Gamma\left[\frac{n}{2}\right]}{\sqrt{\pi } \Gamma\left[\frac{n-1}{2} \right]}\int_0^{\pi/2} \left(\frac{{(1-r^2)\sin^{n-2}\theta} }{(1+r^2-2 r \cos\theta)^{n/2}}-\frac{{(1-r^2)}\sin^{n-2}\theta }{(1+r^2+2 r \cos\theta)^{n/2}}\right)d\theta.$$
 Let $P=2r/(1+r^2)$. Then \[\begin{split}&\frac{{(1-r^2)\sin^{n-2}\theta} }{(1+r^2-2 r \cos\theta)^{n/2}}-\frac{{(1-r^2)}\sin^{n-2}\theta }{(1+r^2+2 r \cos\theta)^{n/2}}\\& = \frac{(1-r^2)}{(1+r^2)^{n/2}}\sum_{k=0}^\infty \left(\binom{-n/2}{k}((-1)^k-1)\cos^k\theta \sin^{n-2}\theta \right)P^k.\end{split}\]
 Since $$\int_0^{\pi/2} \cos^k\theta \sin^{n-2}\theta d\theta = \frac{\Gamma\left[\frac{1+k}{2}\right] \Gamma\left[\frac{1}{2} (-1+n)\right]}{2 \Gamma\left[\frac{k+n}{2}\right]},$$
 we obtain $$U(rN)= \frac{\Gamma\left[\frac{n}{2}\right]}{\sqrt{\pi } \Gamma\left[\frac{n-1}{2} \right]}
 \frac{(1-r^2)}{(1+r^2)^{n/2}}\sum_{k=0}^\infty \frac{\Gamma\left[\frac{1+k}{2}\right] \Gamma\left[\frac{n-1}{2} \right]}{2 \Gamma\left[\frac{k+n}{2}\right]}\binom{-n/2}{k}((-1)^k-1)P^k.$$
Hence $$U(r N) =r \left(1-r^2\right) \left(1+r^2\right)^{-1-\frac{n}{2}} \frac{2 \Gamma\left[1+\frac{n}{2}\right]}{\sqrt{\pi } \Gamma\left[\frac{1+n}{2}\right]} G(r),$$ where
$$G(r)={_{3}\mathrm{F}_{2}}\left[
1,\frac{2+n}{4},\frac{4+n}{4};
\frac{3}{2},\frac{1+n}{2}
; \frac{4 r^2}{\left(1+r^2\right)^2}\right].$$
By \cite[Eq.~3.1.8]{hyper} for $a=\frac{n}{2},\ b=\frac{1}{2} (-1+n),\ c=\frac{1}{2}$, we have that
$$G(r)=\frac{\left(1+r^2\right)^{1+\frac{n}{2}} {_4F_3}\left[\left\{\frac{n}{2},\frac{1}{2} (-1+n),\frac{1}{2},1+\frac{n}{4}\right\},\left\{\frac{n}{4},\frac{3}{2},\frac{1}{2}+\frac{n}{2}\right\},-r^2\right]}{1-r^2}.$$ So

$$U(r N) = r\frac{2 \Gamma\left[1+\frac{n}{2}\right]}{\sqrt{\pi } \Gamma\left[\frac{1+n}{2}\right]}{_4F_3}\left[\left\{\frac{n}{2},\frac{1}{2} (-1+n),\frac{1}{2},1+\frac{n}{4}\right\},\left\{\frac{n}{4},\frac{3}{2},\frac{1}{2}+\frac{n}{2}\right\},-r^2\right],$$ which can be written as

$$U(rN)= \frac{2 \Gamma\left[1+\frac{n}{2}\right]}{\sqrt{\pi } \Gamma\left[\frac{1+n}{2}\right]}r+\sum_{k=1}^\infty \frac{2 (-1)^k (4 k+n) \Gamma\left[k+\frac{n}{2}\right]}{(1+2 k) (-1+2 k+n) \sqrt{\pi } \Gamma[1+k] \Gamma\left[\frac{1}{2} (n-1)\right]}r^{2k+1}.$$
Thus $$\frac{\partial U(r N)}{\partial r} =\frac{2 \Gamma\left[1+\frac{n}{2}\right]}{\sqrt{\pi } \Gamma\left[\frac{1+n}{2}\right]}+\sum_{k=1}^\infty \frac{2 (-1)^k (4 k+n) \Gamma\left[k+\frac{n}{2}\right]}{ (-1+2 k+n) \sqrt{\pi } \Gamma[1+k] \Gamma\left[\frac{1}{2} (n-1)\right]}r^{2k}.$$
Since \[\begin{split}&\frac{2 (-1)^k (4 k+n) \Gamma\left[k+\frac{n}{2}\right]}{ (-1+2 k+n) \sqrt{\pi } \Gamma[1+k] \Gamma\left[\frac{1}{2} (n-1)\right]}\\&=\frac{(-1)^k 2^n \Gamma\left[1+\frac{n}{2}\right] \Gamma\left[k+\frac{n}{2}\right]}{\pi  k! \Gamma[n]}+\frac{2 (-1)^k (-2+n) \Gamma\left[k+\frac{n}{2}\right]}{(-1+2 k+n) \sqrt{\pi } \Gamma[k] \Gamma\left[\frac{1+n}{2}\right]}\end{split}\]
we obtain that
$$\frac{\partial U(r N)}{\partial r} = \frac{\Gamma\left[1+\frac{n}{2}\right]\left((1+r^2)^{-n/2} (1+n)-(n-2) r^2\, _2F_1\left[\frac{1+n}{2},\frac{2+n}{2},\frac{3+n}{2},-r^2\right]\right) }{\sqrt{\pi } \Gamma\left[\frac{3+n}{2}\right]},$$
which in view of the Kummer quadratic transformation, can be written in the form $$\frac{\partial U(r N)}{\partial r}= \frac{\Gamma\left[1+\frac{n}{2}\right](1+r^2)^{-n/2}  \left(1+n-(n-2) r^2 \, _2F_1\left[\frac{1}{2},1,\frac{3+n}{2},-r^2\right]\right)}{\sqrt{\pi }  \Gamma\left[\frac{3+n}{2}\right]}.$$
The function $$y _2F_1[1/2, 1, (3 + n)/2, -y]$$ increases in $y$. Namely its derivative is \[\begin{split}_2F_1[1/2, 2, (3 + n)/2, -y]&=\sum_{m=0}^\infty (-1)^m a(m)y^m\\&=\sum_{m=0}^\infty\frac{(-1)^m (1+m) \Gamma\left[\frac{1}{2}+m\right] \Gamma\left[\frac{3+n}{2}\right]}{\sqrt{\pi } \Gamma\left[\frac{3}{2}+m+\frac{n}{2}\right]} y^m.\end{split}\] Then  $a(m)>0$ and  $$\frac{a(m)}{a(m+1)}=\frac{(1+m) (3+2 m+n)}{(2+m) (1+2 m)}>1$$ because $1 + n + m n>0$, and so  $$_2F_1[1/2, 2, (3 + n)/2, -y] \ge \sum_{m=0}^\infty (a(2m)-a(2m+1))y^{2m}>0.$$ The conclusion is that $\frac{\partial U(r N)}{\partial r}$ is decreasing.
In particular
$$\frac{\partial U(r N)}{\partial r}\ge  \frac{\partial U(r N)}{\partial r}|_{r=1}.$$
For $r=1$ we have
$$\frac{\partial U(r N)}{\partial r}=C_n=\frac{  n! \left(1+n-(n-2)\, _2\mathrm{F}_1\left[\frac{1}{2},1,\frac{3+n}{2},-1\right]\right)}{2^{3 n/2} \Gamma\left[\frac{1+n}{2}\right] \Gamma\left[\frac{3+n}{2}\right]}.$$
\end{proof}

\begin{theorem}\label{Theo1}
If $u$ is a harmonic mapping of the unit ball into itself such that $u(0)=0$, then for $x\in B$  the following sharp inequality $$\frac{1-\|u(x)\|}{1-\|x\|}\ge C_n$$ holds.
\end{theorem}
\begin{proof}
From Theorem~\ref{scp}  we have that $\|u(x)\|\le U(rN)$ and so $$\frac{1-\|u(x)\|}{1-\|x\|}\ge \frac{1-|U(rN)|}{1-\|x\|}.$$

Further there is $\rho\in(r,1)$ such that  $$ \frac{1-U(rN)}{1-\|x\|}= \frac{\partial U(\rho N)}{\partial r},$$ which in view of Lemma~\ref{granit} is bigger that $C_n$. The proof is completed.
\end{proof}
\begin{theorem}\label{Theo2}
a)  If $u$ is a harmonic mapping of the unit ball {\bf into} itself such that $u(0)=0$, and for some $\|\zeta\|=1$ we have $\lim_{r \to 1} \|u(r \zeta)\|=1$ then  \begin{equation}\label{ewo}\liminf_{r\to 1^-}\left\|\frac{\partial u}{\partial \mathbf{n}}(r\zeta)\right\|\ge C_n.\end{equation}

b) If $u$ is a proper harmonic mapping of the unit ball {\bf onto} itself such that $u(0)=0$, then the following sharp inequality \begin{equation}\label{ewo}\liminf_{r\to 1^-}\left\|\frac{\partial u}{\partial \mathbf{n}}(r\zeta)\right\|\ge C_n,\ \  \|\zeta\|=1\end{equation} holds.
Here and in the sequel $\mathbf{n}$ is outward-pointing unit normal.
\end{theorem}
\begin{proof}
Prove a). Then b) follows from a).
Let  $0<r<1$ and $x\in(r\zeta, \zeta)$. There is a $\rho\in(\|x\|,1)$ such that \begin{equation}\label{letw}\frac{1-\|u(x)\|}{1-r}= \frac{\partial \|u(r\zeta )\|}{\partial r}\bigg|_{r=\rho}.\end{equation} On the other hand
$$\left\|\frac{\partial u(r\zeta)}{\partial r}\right\|\ge \frac{\partial \|u(r\zeta)\|}{\partial r}.$$
Letting $\|x\|=r\to 1$, in view of Thereom~\ref{Theo1} and \eqref{letw}, we obtain that $$\liminf_{r\to 1} \left\|\frac{\partial u}{\partial \mathbf{n}}(r\zeta)\right\|\ge C_n.$$
To show that the inequality \eqref{ewo} is sharp, let $$h_m(x) =\left\{
                                                                 \begin{array}{ll}
                                                                   1-x/m, & \hbox{if $x\in(1/m,1]$;} \\
                                                                   (m-1)x, & \hbox{if $-1/m\le x\le 1/m$;} \\
                                                                   -1-x/m, & \hbox{if $x\in [-1,-1/m)$,}
                                                                 \end{array}
                                                               \right.$$ and define $$f_m(x_1,\dots,x_{n-1},x_n) = \frac{\sqrt{1-h_m(x_n)^2}}{\sqrt{1-x_n^2}}(x_1,\dots,x_{n-1},0)+(0,\dots,0, h_m(x_n)).$$ Then $f_m$ is a homeomorphism of the unit sphere onto itself, such that $$\lim_{m\to \infty} f_m(x) = (0,\dots,0,\chi_{S^+}(x)-\chi_{S^-}(x)).$$ Further  $u_m(x) = \mathcal{P}[f_m](x)$ is a harmonic mapping of the unit ball onto itself such that $\lim_{\|x\|\to 1}\|u_m(x)\|=1$. Thus $u_m$ is proper. Moreover $u_m(0)=0$ and   $\lim_{m\to \infty} u_m(x) = (0,\dots, 0, U(x)).$ This implies the fact that the constant $C_n$ is sharp.
\end{proof}
\begin{remark}

The following table shows  first few constants $C_n$ and related functions

\begin{tabular}{|c|c|c|c|}
  \hline
  $n$&  $u(rN)$ & $\partial_r u(rN)$ & $C_n$  \\
     \hline
 $2$ & $\frac{4\arctan(r)}{\pi }$ & $\frac{4}{\pi(1+r^2)},$ & $\frac{2}{\pi}$  \\
 \hline
  $3$ & $\frac{-1+r^2+\sqrt{1+r^2}}{r \sqrt{1+r^2}}$ & $\frac{1-\sqrt{1+r^2}-r^2 \left(-3+\sqrt{1+r^2}\right)}{r^2 \left(1+r^2\right)^{3/2}}$ &$\sqrt{2}-1$ \\
   \hline
  $4$ & $\frac{2 r \left(-1+r^2\right)+2\left(1+r^2\right)^2 \arctan r}{\pi  r^2 \left(1+r^2\right)}$& $ \frac{4 \left(r+3 r^3-\left(1+r^2\right)^2 \arctan r\right)}{\pi  r^3 \left(1+r^2\right)^2}$ & $\frac{4-\pi }{\pi }$ \\
  \hline
\end{tabular}

\end{remark}

\end{document}